\documentclass[12pt,reqno]{amsart}

\usepackage{amssymb}

\newtheorem{theorem}{Theorem}[section]
\newtheorem{lemma}[theorem]{Lemma}
\newtheorem{proposition}[theorem]{Proposition}
\newtheorem{corollary}[theorem]{Corollary}

\theoremstyle{remark}
\newtheorem{remark}[theorem]{Remark}

\numberwithin{equation}{section}

\begin{document}

\title[Monodromy map from differential systems to character variety]{The monodromy map
from differential systems to character variety is generically immersive}

\author[I. Biswas]{Indranil Biswas}

\address{School of Mathematics, Tata Institute of Fundamental Research,
Homi Bhabha Road, Mumbai 400005, India}

\email{indranil@math.tifr.res.in}

\author[S. Dumitrescu]{Sorin Dumitrescu}

\address{Universit\'e C\^ote d'Azur, CNRS, LJAD, France}

\email{dumitres@unice.fr}



\date{}

\begin{abstract}
Let $G$ be a connected reductive affine algebraic group defined over $\mathbb C$ and $\mathfrak g$ 
its Lie algebra. We consider all pairs of the form $(Y,\, D)$, where $Y$ is a complex structure on a compact
oriented $C^\infty$ surface $\Sigma$, and $D$ is a holomorphic connection on the trivial holomorphic principal $G$--bundle
$Y\times G$ on $Y$; these are known as $\mathfrak g$--differential systems.
We study the monodromy map from the space of $\mathfrak g$--differential systems 
to the character variety of $G$--representations of the fundamental group of $\Sigma$. If the complex dimension
of $G$ is at least 
three, and $\text{genus}(\Sigma)\, \geq \, 2$, we show that the monodromy map is an immersion at the generic point.
\end{abstract}

\maketitle

\tableofcontents

\section{Introduction}\label{sec0}

Our aim here is to study the monodromy map, also called the Riemann--Hilbert map, from the
differential systems over compact Riemann surfaces to the character varieties. In order to 
describe the framework, let us denote by $\Sigma$ a given compact connected oriented 
$C^\infty$ surface of genus $g \,\geq\, 2$ and by $G$ a connected reductive affine 
algebraic group defined over $\mathbb C$.

Consider a complex structure $X$ on $\Sigma$ (it gives an element in the Teichm\"uller space for
$\Sigma$) and a holomorphic (flat) connection $\phi$ on the trivial holomorphic principal 
$G$--bundle $X \times G$ over $X$. Recall that $\phi$ is determined by an element $\delta\, \in\, 
H^0(X,\, K_X)\otimes{\mathfrak g},$ where $\mathfrak g$ is the (complex) Lie algebra of 
$G$ and $K_X$ is the canonical line bundle of $X$.
Fixing a base point $x_0\, \in\, X$, consider the corresponding
universal cover $ \pi \,:\, \widetilde{X} \,\longrightarrow\, X$ of $X$, and
endow the trivial principal $G$--bundle $\widetilde{X}\times G$ over $\widetilde X$
with the pulled back holomorphic flat connection $\pi^*\phi$.

For any locally defined parallel section $s$ of $X \times G$ for the connection $\phi$, 
the pulled back local section $\pi^*s$ of $\widetilde{X}\times G$ extends to a 
$\pi^*\phi$--parallel section over entire $\widetilde X$. This extension of $\pi^*s$ 
produces a holomorphic map $\widetilde{X} \,\longrightarrow\, G$ which is $\pi_1(X,\, 
x_0)$--equivariant with respect to the natural action of $\pi_1(X,\, x_0)$ on $\widetilde 
X$ through deck transformations and the action of $\pi_1(X,\, x_0)$ on $G$ through a group 
homomorphism $\pi_1(X,\, x_0) \,\longrightarrow\, G$; this homomorphism $\pi_1(X,\, x_0) 
\,\longrightarrow\, G$ is known as the {\it monodromy } of the flat connection $\phi$.

Although the above mentioned monodromy homomorphism depends on the choice of the 
holomorphic trivialization of the principal $G$--bundle, the element of the character 
variety of $G$-representations $$\Xi \,:=\, \text{Hom}(\pi_1(X),\, G)/\!\!/G$$ given by it 
is independent of both the trivialization of the principal $G$--bundle and the base point 
$x_0$ (and also of the choice of $s$). It should be clarified that the monodromy map is 
defined from the isomorphism classes of flat principal $G$--bundles on $X$ to $\Xi$; to 
define this map the underlying holomorphic principal $G$--bundle is not needed to be 
trivial.

Recall that $\Xi$ is a (singular) complex analytic space of dimension $$2((g-1)\cdot\dim\, 
[G,\, G] + g\cdot (\dim G - \dim\, [G,\, G]));$$ see, for example, \cite{Go}, 
\cite[Proposition 49]{Si}. We shall denote the dimension of the commutator group $\dim 
[G,\, G]$ by $d$ and $\dim G - \dim\, [G,\, G]$ by $c$ (it is the dimension of the center 
of $G$). With this notation, the complex dimension of $\Xi$ is $2(g-1)d+2gc$.

Let us adopt the notation of \cite{CDHL} and denote by $\text{Syst}$ the space of all pairs $(X,\, 
\phi)$, where $X$ is an element of the Teichm\"uller space for $\Sigma$ and $\phi$ is a holomorphic 
connection on the trivial principal $G$--bundle $X \times G$ over $X$; recall that a holomorphic connection on 
a holomorphic bundle over a Riemann surface is automatically flat. This space of all differential
systems on $\Sigma$ is a complex space of dimension $(g-1)(d+3)+gc$.

Sending a holomorphic connection to its monodromy representation, 
a holomorphic mapping
$$\widetilde{\text{Mon}}\, :\, \text{Syst}\, \longrightarrow\, \Xi$$
is obtained; in other words, $\widetilde{\text{Mon}}$ is the restriction to $\text{Syst}$
of the Riemann--Hilbert map.

Let us define the nonempty Zariski open subset $\text{Syst}^{\text{irred}}$ of $\text{Syst}$ 
consisting of all pairs $(X,\, \phi)$ for which the connection $\phi$ is {\it irreducible}, meaning 
the monodromy homomorphism for $\phi$ does not factor through any proper parabolic subgroup of $G$.
It is a connected complex orbifold of dimension $(g-1)(d+3)+gc$ (see Lemma \ref{lemsmm}).
The image of $\text{Syst}^{\text irred}$ under the above map $\widetilde{\text{Mon}}$,
that sends a holomorphic connection to the corresponding monodromy representation, lies in the 
Zariski open subset $$\Xi^{\text irred}\, \subset\, \Xi$$ defined by the irreducible 
homomorphisms $\pi_1(X)\, \longrightarrow\, G$ (i.e., the homomorphisms that do not factor through some 
proper parabolic subgroup of $G$). Although the complex space $\Xi^{\text irred}$ is not smooth in general,
its singularities are finite group quotients. Let
\begin{equation}\label{mm}
\text{Mon} \,:\, \text{Syst}^{\text{irred}} \,\longrightarrow\, \Xi^{\text irred}
\end{equation}
be the holomorphic map between complex manifolds given by 
the restriction of $\widetilde{\text{Mon}}$ to $\text{Syst}^{\text{irred}}\, \subset\,
\text{Syst}$.

The main result proved here is the following (see Theorem \ref{thm2}): 

\begin{theorem}\label{thm principal}
If the complex dimension of $G$ is at least three, the monodromy map
$${\rm Mon} \,:\, {\rm Syst}^{\rm irred} \, \longrightarrow\,
\Xi^{\rm irred}$$ in \eqref{mm} is an immersion at the generic point.
\end{theorem}

Remark \ref{remla} explains that the assumption $\dim G \, \geq\, 3$ in Theorem \ref{thm principal} is necessary.

If $G\,=\, {\rm SL}(2,{\mathbb C})$, then the dimensions of
$\text{Syst}^{\text irred} $ and $\Xi^{\text irred}$ are both $6g-6$,
and Theorem \ref{thm principal} implies that
$\text{Mon}$ in \eqref{mm} is a
local biholomorphism at the generic point (see Corollary \ref{cor3}).
It should be mentioned that examples constructed in \cite{CDHL} show that for
$G\,=\, {\rm SL}(2,{\mathbb C})$ and $\Sigma$ of genus $g \,\geq\, 3$, the monodromy map $\text{Mon}$ 
is not always a local biholomorphism (over entire $\text{Syst}^{\text{irred}}$).

When $G\,=\, {\rm SL}(2,{\mathbb C})$ and $g\,=\,2$, the main result of \cite{CDHL} says that the 
map $\text{Mon}$ in \eqref{mm} is a local biholomorphism over entire $\text{Syst}^{\text{irred}}$. 
An alternative proof of this result of \cite{CDHL} is given in Corollary \ref{cor2}. In this 
context it should be mentioned that our work was greatly influenced by \cite{CDHL}.

Just as for the authors of \cite{CDHL}, our main motivation came from a question of E. 
Ghys for $G\,=\, {\rm SL}(2,{\mathbb C})$ relating the monodromy of ${\mathfrak 
s}{\mathfrak l}(2, {\mathbb C})$--differential systems to the existence of holomorphic 
curves of genus $g \,>\,1$ lying in compact quotients of ${\rm SL}(2,{\mathbb C})$ by 
lattices $\Gamma$. Such compact quotients of ${\rm SL}(2,{\mathbb C})$ are non-K\"ahler 
manifolds. These non-K\"ahler manifolds do not admit any closed complex hypersurface \cite{HM}. 
It is known that elliptic curves do exist in some of those manifolds, but the existence of 
holomorphic curves of genus $g\,>\,1$ is still an open question. E. Ghys realized that 
constructing an irreducible ${\mathfrak s}{\mathfrak l}(2,{\mathbb C})$--differential 
system on a Riemann surface $X$ with monodromy lying inside a cocompact lattice $\Gamma\, 
\subset\, {\rm SL}(2,{\mathbb C})$ would provide a nontrivial holomorphic map from $X$ 
into the quotient of ${\rm SL}(2,{\mathbb C})/\Gamma$ (in fact the two problems are 
equivalent). While the question asked by Ghys is still open, the above Theorem \ref{thm 
principal} extends the results of \cite{CDHL} and leads to an enhancement of the 
understanding of the monodromy of the differential systems.

The strategy of the proof of Theorem \ref{thm principal} and the organization of the paper 
are as follows. We consider the monodromy map, to the character variety, defined on the 
space of triples $(X,\, E_G,\, \phi)$, with $X$ an element of the Teichm\"uller space for 
$\Sigma$, $E_G$ a holomorphic principal $G$-bundle over $X$ and $\phi$ a holomorphic 
connection on $E_G$. In Section \ref{section 2} we define a $2$-term complex ${\mathcal 
C}_\bullet$ over $X$ whose first hypercohomology gives the infinitesimal deformations of 
$(X,\, E_G, \,\phi)$ (see Theorem \ref{thm1} (2)). Moreover, the kernel of the differential of 
the monodromy map coincides with the image of the space of deformations of the complex 
structure $H^1(X,\, TX)$ through a certain homomorphism $\beta_\phi$ from $H^1(X,\, TX)$ 
to the $1$-hypercohomology of ${\mathcal C}_\bullet$ (see Theorem \ref{thm1} (4)). In Section 
\ref{section 3} we fix $E_G$ to be the holomorphically trivial principal $G$--bundle over 
$X$ and set $\phi$ to be an irreducible holomorphic connection on it. We show that the 
tangent space of $\text{Syst}^{\text irred}$ at $(X,\, \phi)$, which is naturally embedded 
in ${\mathbb H}^1({X, \mathcal C}_\bullet)$, is transverse to the kernel of the monodromy 
map, provided $\phi$ satisfies a geometric criterion described in Proposition \ref{prop1}. 
In Section \ref{section 4} we consider the special case of $G\,=\, {\rm SL}(2,{\mathbb 
C})$ and we prove that the criterion in Proposition \ref{prop1} is satisfied at any point 
$(X,\, \phi) \,\in\, \text{Syst}^{\text{irred}}$ for surfaces $\Sigma$ of genus two (see 
Proposition \ref{prop2}); the same holds for the generic point in 
$\text{Syst}^{\text{irred}}$ for surfaces $\Sigma$ of genus three (see Lemma \ref{lem1}). 
The main result (Theorem \ref{thm principal}) is obtained in Section \ref{section 5} where 
Lemma \ref{lem2} proves that the transversality criterion is satisfied at the generic 
point in $\text{Syst}^{\text{irred}}$. More precisely, the proof of Lemma \ref{lem2} shows 
that the transversality criterion (in Proposition \ref{prop1}) is implied by the statement 
that for a non-hyperelliptic Riemann surface $X$, and a generic three dimensional subspace 
$W\, \subset\, H^0(X,\, K_X)$, the natural homomorphism $$\Theta_W\, :\, H^0(X,\, 
K_X)\otimes W\, \longrightarrow\, H^0(X,\, K^2_X) $$ is surjective. The above statement is 
precisely the Theorem 1.1 in \cite[p.~221]{Gi}, where the proof of it is attributed to R. 
Lazarsfeld.

\section{Infinitesimal deformations of bundles and connections}\label{section 2} 

In this section we introduce several infinitesimal deformation spaces and natural morphisms between them.

The holomorphic tangent bundle of a complex manifold $Y$ will be denoted by $TY$.

Let $X$ be a compact connected Riemann surface. The holomorphic
cotangent bundle of $X$ will be denoted by $K_X$.
Let $G$ be a connected reductive affine algebraic group defined over $\mathbb C$.
The Lie algebra of $G$ will be denoted by $\mathfrak g$.

Take a holomorphic principal $G$--bundle over $X$
\begin{equation}\label{e1}
p\, :\, E_G\,\longrightarrow\, X\, .
\end{equation}
So $E_G$ is equipped with a holomorphic action of $G$ on the right which is both free
and transitive on the fibers of $p$, and furthermore, $E_G/G\,=\, X$. Consider
the holomorphic right action of $G$ on the holomorphic tangent bundle
$TE_G$ given by the action of $G$ on $E_G$. The quotient
$$
\text{At}(E_G)\,:=\, (TE_G)/G
$$
is a holomorphic vector bundle over $E_G/G\,=\, X$; it is called the
\textit{Atiyah bundle} for $E_G$. The differential $$dp\, :\,
TE_G\, \longrightarrow\, p^* TX$$ of the projection $p$ in \eqref{e1} is $G$--equivariant
for the trivial action of $G$ on the fibers of $p^*TX$. The action of $G$ on $E_G$ produces
a holomorphic homomorphism from the trivial holomorphic bundle
$$
E_G\times {\mathfrak g}\, \longrightarrow\, \text{kernel}(dp)
$$
which is an isomorphism. Therefore, we have a short
exact sequence of holomorphic vector bundles on $E_G$
\begin{equation}\label{e1p}
0\,\longrightarrow\, \text{kernel}(dp)\,=\,E_G\times{\mathfrak g}\,\longrightarrow
\,\text{At}(E_G)\, \stackrel{dp}{\longrightarrow}\, p^*TX\, \longrightarrow\, 0
\end{equation}
in which all the homomorphisms are $G$--equivariant. The quotient $\text{kernel}(dp)/G$ is the
adjoint vector bundle $\text{ad}(E_G)\,=\, E_G({\mathfrak g})$, which is the holomorphic
vector bundle over $X$ associated to $E_G$ for the adjoint action of $G$ on $\mathfrak g$. Taking
quotient of the bundles in \eqref{e1p}, by the actions of $G$, the following short exact sequence of
holomorphic vector bundles on $X$ is obtained:
\begin{equation}\label{e2}
0\, \longrightarrow\, \text{ad}(E_G)\,\stackrel{\iota}{\longrightarrow}\, \text{At}(E_G)\,
\stackrel{d'p}{\longrightarrow}\, TX \,\longrightarrow\, 0
\end{equation}
\cite{At}; it is known as the Atiyah exact sequence for $E_G$.

A holomorphic connection on $E_G$ is a holomorphic homomorphism of vector bundles
$$
\phi\, :\, TX\, \longrightarrow\, \text{At}(E_G)
$$
such that
\begin{equation}\label{e3}
(d'p)\circ\phi\,=\, \text{Id}_{TX}\, ,
\end{equation}
where $d'p$ is the projection in \eqref{e2} (see \cite{At}). A holomorphic connection on a
holomorphic bundle
over $X$ is automatically flat, because $\Omega^{2,0}_X\,=\, 0$. A holomorphic connection $\phi$
on $E_G$ gives a holomorphic decomposition $\text{At}(E_G)\,=\, TX\oplus \text{ad}(E_G)$ into
a direct sum of holomorphic vector bundles. This
decomposition produces a holomorphic homomorphism
\begin{equation}\label{e4}
\phi'\, :\, \text{At}(E_G)\, \longrightarrow\, \text{ad}(E_G)
\end{equation}
such that $\phi'\circ\iota\,=\, \text{Id}_{\text{ad}(E_G)}$, where $\iota$ is the homomorphism
in \eqref{e2}.

Take a holomorphic connection
\begin{equation}\label{f1}
\phi\, :\, TX\, \longrightarrow\, \text{At}(E_G)
\end{equation}
on $E_G$. Since $\text{At}(E_G)\,=\, (TE_G)/G$, this homomorphism $\phi$ produces
a $G$--equivariant holomorphic homomorphism of vector bundles
\begin{equation}\label{wp}
\widehat{\phi}\, :=\, p^*\phi \, :\, p^*TX\, \longrightarrow\, TE_G
\end{equation}
over $E_G$. Take any analytic open subset $U\, \subset\, X$. Let $s$ be a holomorphic section
of $\text{At}(E_G)\vert_U$ over $U$. Since $\text{At}(E_G)\,=\, (TE_G)/G$, we have
$$
\widehat{s}\, :=\, p^*s\, \in\, H^0(p^{-1}(U),\, TE_G)^G\, \subset\,
H^0(p^{-1}(U),\, TE_G)\, .
$$
For any holomorphic vector field $t\, \in\, H^0(U,\, TU)$, consider the Lie bracket
$$
[\widehat{\phi}(p^*t),\, \widehat{s}]\, \in\, H^0(p^{-1}(U),\, TE_G)\, ,
$$
where $\widehat{\phi}$ is the homomorphism in \eqref{wp}.
This vector field $[\widehat{\phi}(p^*t),\, \widehat{s}]$ on $p^{-1}(U)$ is
$G$--invariant, because both $\widehat{s}$ and $\widehat{\phi}(p^*t)$ are so.
Therefore, $[\widehat{\phi}(p^*t),\, \widehat{s}]$ produces a holomorphic section of
$\text{At}(E_G)$ over $U$; this section of $\text{At}(E_G)\vert_U$ will be denoted
by $A(t, \, s)$. Let
\begin{equation}\label{php}
\phi'(A(t,\, s))\, \in \, H^0(U,\, \text{ad}(E_G))
\end{equation}
be the section of $\text{ad}(E_G)\vert_U$, where $\phi'$ is the projection in \eqref{e4}.

Now, for any holomorphic function $f$ defined on $U$, we have
\begin{equation}\label{h1}
[\widehat{\phi}(p^*(f\cdot t)),\, \widehat{s}]\,=\,
(f\circ p)\cdot [\widehat{\phi}(p^*t),\, \widehat{s}]- 
\widehat{s}(f\circ p)\cdot \widehat{\phi}(p^*t)\, .
\end{equation}
Since $\phi'(\widehat{\phi}(p^*t))\,=\, 0$, where $\phi'$ is constructed
in \eqref{e4}, from \eqref{h1} it follows immediately that
$$
\phi'(A(f\cdot t,\, s))\,=\, f\cdot \phi'(A(t,\, s))\, ;
$$
$\phi'(A(t,\, s))$ is defined in \eqref{php}. Let
\begin{equation}\label{e6}
\Phi\, :\, \text{At}(E_G) \, \longrightarrow\, \text{ad}(E_G)\otimes K_X
\end{equation}
be the homomorphism of sheaves defined by the equation
$$
\langle \Phi(s),\, t\rangle \,=\, \phi'(A(t,\, s))\, \in\, H^0(U,\, \text{ad}(E_G)) \, ,
$$
where $s$ and $t$ are holomorphic sections, over $U$, of $\text{At}(E_G)$ and
$TX$ respectively, while $\langle -\, -\rangle$ is the contraction of $K_X$ by $T_X$.

\begin{remark}\label{rem-r1}
It should be mentioned that the map $\Phi$ in \eqref{e6} is an
additive homomorphism and it is $\mathbb 
C$--linear, but it is not ${\mathcal O}_X$--linear. In fact, the composition of 
homomorphisms
$$\Phi\circ\iota\, :\, \text{ad}(E_G)\, \longrightarrow\, \text{ad}(E_G)\otimes K_X,
$$
where $\iota$ is the inclusion map in \eqref{e2}, satisfies the Leibniz identity.
This map $\Phi\circ\iota$ is the connection on $\text{ad}(E_G)$ induced by the
connection $\phi$ on $E_G$.
\end{remark}

The composition
$$
\Phi\circ\phi\, :\, TX\, \longrightarrow\, \text{ad}(E_G)\otimes K_X
$$
coincides with the curvature of the connection $\phi$. Since $\phi$ is flat, we have
\begin{equation}\label{e7}
\Phi\circ\phi\, =\, 0\, .
\end{equation}

Let ${\mathcal C}_{\bullet}$ be the $2$--term complex
$$
{\mathcal C}_{\bullet} \, :\, {\mathcal C}_0\,:=\, \text{At}(E_G) \,
\stackrel{\Phi}{\longrightarrow}\, {\mathcal C}_1 \,:=\, \text{ad}(E_G)\otimes K_X\, ,
$$
where ${\mathcal C}_i$ is at the $i$--th position and $\Phi$ is the
$\mathbb C$--linear additive homomorphism constructed in \eqref{e6}.
Using \eqref{e7} we have the following commutative diagram of
homomorphisms of complexes of sheaves on $X$:
\begin{equation}\label{e9}
\begin{matrix}
&& 0 & & 0\\
&& \Big\downarrow & & \Big\downarrow \\
&& TX & \longrightarrow & 0\\
&& \,\,\,\, \Big\downarrow\phi && \Big\downarrow\\
{\mathcal C}_\bullet &: & {\mathcal C}_0 & \stackrel{\Phi}{\longrightarrow} & {\mathcal C}_1\\
&& \,\,\,\,\,\,\,\Big\downarrow= && \Big\downarrow\\
&& \text{At}(E_G) & \longrightarrow & 0\\
& & \Big\downarrow & & \Big\downarrow \\
& & 0 & & 0
\end{matrix}
\end{equation}
It should be clarified that this is not a complex of complexes of sheaves --- the composition map
does not vanish, because $\phi\, \not=\, 0$. Let
\begin{equation}\label{e10}
H^1(X,\, TX) \, \stackrel{\beta_\phi}{\longrightarrow}\, {\mathbb H}^1({X, \mathcal C}_\bullet)
\, \stackrel{\gamma}{\longrightarrow}\, H^1(X,\, \text{At}(E_G))
\end{equation}
be the homomorphisms of (hyper)cohomologies associated to the
homomorphisms in \eqref{e9}, where ${\mathbb H}^i$ denotes the
$i$--th hypercohomology. It should be clarified that $\gamma\circ\beta_\phi$ does not vanish. Indeed,
the composition of $\gamma\circ\beta_\phi$ with the homomorphism $$(d'p)_*\, :\, H^1(X,\, \text{At}(E_G))\,
\longrightarrow \, H^1(X,\, TX)\, ,$$ where $d'p$ is the projection in \eqref{e2}, coincides with the identity
map of $H^1(X,\, TX)$.

The following known theorem will be used (see \cite{Ch2}, \cite{Don}, \cite{In}, 
\cite{BHH}).

\begin{theorem}\label{thm1}
\mbox{}
\begin{enumerate}
\item The infinitesimal deformations of the pair $(X,\, E_G)$ are parametrized by
the elements of the cohomology $H^1(X,\, {\rm At}(E_G))$.

\item The infinitesimal deformations of the triple $(X,\, E_G,\, \phi)$ are parametrized by
the elements of the hypercohomology ${\mathbb H}^1(X,\, {\mathcal C}_\bullet)$.

\item The forgetful map from the 
infinitesimal deformations of the triple $(X,\, E_G,\, \phi)$ to the
infinitesimal deformations of the pair $(X,\, E_G)$, that forgets the connection $\phi$,
is the homomorphism $\gamma$ in \eqref{e10}.

\item The infinitesimal isomonodromy map, from the infinitesimal deformations of $X$ to the
infinitesimal deformations of the triple $(X,\, E_G,\, \phi)$, coincides with the homomorphism
$\beta_\phi$ in \eqref{e10}.
\end{enumerate}
\end{theorem}

For the proof of Theorem \ref{thm1} the reader is referred to \cite[p.~1413, Proposition 4.3]{Ch2} 
(for a proof of Theorem \ref{thm1}(1)), \cite[p.~ 1415, Proposition 4.4]{Ch2} (for
a proof of Theorem \ref{thm1}(2)) and 
\cite[p.~1417, Proposition 5.1]{Ch2} (for a proof of Theorem \ref{thm1}(4)); see also \cite{Ch1}.
Theorem \ref{thm1}(3) is evident.

The Atiyah exact sequence in \eqref{e2} produces a long exact sequence of cohomologies
$$
H^1(X,\, \text{ad}(E_G))\,\stackrel{\iota_*}{\longrightarrow}\, H^1(X,\, \text{At}(E_G))\,
\stackrel{(d'p)_*}{\longrightarrow}\, H^1(X,\, TX) \,\longrightarrow\, 0\, .
$$
We note that the infinitesimal deformations of $E_G$ (keeping the Riemann surface $X$ fixed) are 
parametrized by $H^1(X,\, \text{ad}(E_G))$ (see \cite{Don}), and the above homomorphism $\iota_*$ 
coincides with the natural homomorphism of infinitesimal deformations. The above projection 
$(d'p)_*$ is the forgetful map that sends the infinitesimal deformations of the pair $(X,\, E_G)$ 
to the infinitesimal deformations of $X$ that forgets the principal $G$--bundle.

\section{Infinitesimal deformations of connections on the trivial bundle}\label{section 3}

Recall the moduli space of differential systems $\text{Syst}$ and its subset of irreducible differential systems
$\text{Syst}^{\text{irred}}$, both defined in Section \ref{sec0}.

In this section we realize the tangent space to $(X,\, \phi) \,\in\, \text{Syst}^{\text{irred}}$ as 
a subspace embedded in ${\mathbb H}^1(X,\, {\mathcal C}_\bullet)$, the space of infinitesimal deformations of 
triples $(X,\, E_G,\, \phi)$ (see below), and we prove a criterion for transversality to the kernel of the 
monodromy map (Proposition \ref{prop1}).

Let $Y$ be a compact connected Riemann surface, and let $\psi$ be a holomorphic connection on the 
holomorphically trivial principal $G$--bundle $Y\times G\, \longrightarrow\, Y$ over $Y$. Let
\begin{equation}\label{e11}
{\mathcal T}(Y,\, \psi)
\end{equation}
denote the infinitesimal deformations of the pair $(Y,\, \psi)$ (keeping the underlying holomorphic
principal $G$--bundle to be the trivial principal $G$--bundle on the moving Riemann surface).

Henceforth, we assume that $\text{genus}(X)\,=\, g\, \geq\, 2$.

Let $Z_G\, \subset\, G$ be the center of $G$.
As before, $d\, :=\, \dim [G,\, G]$ and $c\, :=\, \dim Z_G$.

\begin{lemma}\label{lemsmm}
The moduli space $\text{Syst}^{\text{irred}}$ in Section \ref{sec0} is a connected (smooth) complex
orbifold. The complex dimension of ${\rm Syst}$ is $(g-1)(d+3)+gc$.
\end{lemma}

\begin{proof}
Let ${\mathbb T}_g$ denote the Teichm\"uller space for genus $g$ Riemann surfaces. We have an universal
family of genus $g$ Riemann surfaces 
$$
\varphi\, :\, {\mathbf C}_g\, \longrightarrow\, {\mathbb T}_g\, .
$$
Let $\Omega_\varphi\, \longrightarrow\, {\mathbf C}_g$ be the relative holomorphic cotangent bundle for the
projection $\varphi$. Consider the direct image
$$
{\mathcal W}\, :=\, \varphi_* \Omega_\varphi\, \longrightarrow\, {\mathbb T}_g\, .
$$
So ${\mathcal W}$ is a holomorphic vector bundle over ${\mathbb T}_g$ whose fiber over any
given Riemann surface $Y\, \in\, {\mathbb T}_g$ is $H^0(Y,\, K_Y)$. Now define the holomorphic
vector bundle ${\mathcal W}(G)$ over ${\mathbb T}_g$
$$
{\mathcal W}(G)\, :=\, {\mathcal W}\otimes_{\mathbb C} {\mathfrak g}\, .
$$
The adjoint action of the group $G$ on ${\mathfrak g}$ and the trivial action of
$G$ on ${\mathcal W}$ together produce an action of $G$ on ${\mathcal W}(G)$. Note that this action
of $G$ on ${\mathcal W}$ factors through the quotient $G/Z_G$ of $G$. Let
$$
\widehat{\mathcal W}(G)\, \subset\, {\mathcal W}(G)
$$
be the open subset consisting of all points $(Y,\, \theta)\, \in\, {\mathcal W}(G)$,
where $Y\, \in\, {\mathbb T}_g$ and $$\theta\, \in\, {\mathcal W}(G)_Y \,=\, H^0(Y,\, K_Y)
\otimes{\mathfrak g},$$
such that $\theta$ is not contained in $H^0(Y,\, K_Y)\otimes{\mathfrak p}$ for some parabolic
subgroup ${\mathfrak p}\, \subsetneq\, {\mathfrak g}$. For any
$(Y,\, \theta)\, \in\, \widehat{\mathcal W}(G)$, consider the holomorphic connection
$D^Y_0+\theta$, where $D^Y_0$ is the trivial connection on the
trivial holomorphic principal $G$--bundle $Y\times G\, \longrightarrow\, Y$. Let
$$
\Psi_\theta\, :\, \pi_1(Y, y_0)\, \longrightarrow\, G
$$
be the monodromy representation for this flat connection $D^Y_0+\theta$. The definition of
$\widehat{\mathcal W}(G)$ ensures that $\Psi_\theta(\pi_1(Y, y_0))$ is not contained in some
proper parabolic subgroup of $G$.

The action of $G/Z_G$ on ${\mathcal W}(G)$ evidently preserves $\widehat{\mathcal W}(G)$. For any
$(Y,\, \theta)\, \in\, \widehat{\mathcal W}(G)$, the isotropy subgroup of $(Y,\, \theta)$, for the
action of $G/Z_G$ on ${\mathcal W}(G)$, is $$N(\Psi_\theta(\pi_1(Y, y_0))')/\Psi_\theta(\pi_1(Y, y_0))'\, ,$$
where $\Psi_\theta(\pi_1(Y, y_0))'$ is the image of $\Psi_\theta(\pi_1(Y, y_0))$ in $G/Z_G$ and
$N(\Psi_\theta(\pi_1(Y, y_0))')\, \subset\, G/Z_G$ is its normalizer. From the
definition of $\widehat{\mathcal W}(G)$ it follows that this isotropy subgroup is finite.

For any Riemann surface $Y\, \in\, {\mathbb T}_g$, the space of all holomorphic connections on the
trivial holomorphic principal $G$--bundle $Y\times G\, \longrightarrow\, Y$ is an affine space for the vector space
${\mathcal W}(G)_Y$, where ${\mathcal W}(G)_Y$ is the fiber of ${\mathcal W}(G)$ over the point $Y$
(see \eqref{f3}); note that $\text{ad}(Y\times G)\,=\, Y\times{\mathfrak g}$.
Consequently, we have a biholomorphism from the quotient space $\widehat{\mathcal W}(G)/G$
$$
\eta\, :\, \widehat{\mathcal W}(G)/G\, \stackrel{\sim}{\longrightarrow}\, {\rm Syst}^{\text{irred}}
$$
that sends any $\theta \, \in\, \widehat{\mathcal W}(G)_Y$, $Y\, \in\, {\mathbb T}_g$, to
the holomorphic connection $D^Y_0+\theta$, where $D^Y_0$ as before is the trivial connection on the
trivial holomorphic principal $G$--bundle $Y\times G\, \longrightarrow\, Y$; note that
$D^Y_0$ does not depend on the choice of the holomorphic trivialization of the
principal $G$--bundle (see the paragraph following \eqref{e12}). Also, this map evidently factors
through the quotient $\widehat{\mathcal W}(G)/G$. In view of the
biholomorphism $\eta$ we conclude that ${\rm Syst}^{\text{irred}}$ is a connected complex orbifold.

The complex dimension of ${\rm Syst}^{\text{irred}}$ is $(g-1)(d+3)+gc$, because the complex dimension of the total
space of ${\mathcal W}(G)$ is $3(g-1)+dg+gc$, and hence the complex dimension of the total
space of $\widehat{\mathcal W}(G)$ is $(g-1)(d+3)+gc$.
\end{proof}

Now take $E_G$ in \eqref{e1} to be the holomorphically trivial principal $G$--bundle
$X\times G$ on $X$. As in \eqref{f1}, take a holomorphic connection $\phi$ on $E_G\,=\,
X\times G$. Since the infinitesimal deformations of the triple $(X,\, E_G,\, \phi)$
are parametrized by ${\mathbb H}^1(X,\, {\mathcal C}_\bullet)$ (see Theorem \ref{thm1}(2)),
we have a natural homomorphism
$$
{\mathcal T}(X,\, \phi)\, \longrightarrow\, {\mathbb H}^1(X,\, {\mathcal C}_\bullet)\, ,
$$
where ${\mathcal T}(X,\, \phi)$ is defined in \eqref{e11}. Let
\begin{equation}\label{e12}
{\mathcal S}(X,\, \phi)\, \subseteq\, {\mathbb H}^1(X,\, {\mathcal C}_\bullet)
\end{equation}
be the image of this homomorphism from ${\mathcal T}(X,\, \phi)$.

The trivial holomorphic principal $G$--bundle $E_G\,=\, X\times G$ has a unique holomorphic
connection whose monodromy is the trivial representation. Once we fix an isomorphism
of $E_G$ with $X\times G$, the trivial holomorphic connection on $X\times G$ induces
a holomorphic connection on $E_G$ using the chosen isomorphism. However, this induced connection
on $E_G$ does not depend on the choice of the trivialization of $E_G$; this
unique connection on $E_G$ will be called the trivial connection. The monodromy
of the trivial connection is evidently trivial.

Using the trivial connection on $E_G$, we have a canonical holomorphic decomposition
\begin{equation}\label{f2}
\text{At}(E_G)\,=\, \text{ad}(E_G)\oplus TX\, .
\end{equation}
Using \eqref{f2}, the holomorphic connections on $E_G$ are identified with holomorphic homomorphisms
from $TX$ to $\text{ad}(E_G)$. More precisely, to any holomorphic homomorphism
\begin{equation}\label{e13}
\rho\, :\, 
TX\, \longrightarrow\, \text{ad}(E_G)\, 
\end{equation}
we assign the corresponding homomorphism
\begin{equation}\label{f3}
\widehat{\rho}\, :\, TX\, \longrightarrow\, \text{ad}(E_G)\oplus TX\, =\, \text{At}(E_G)\, ,\ \
v\, \longmapsto\, (\rho(v),\, v)
\end{equation}
clearly $\widehat{\rho}$ satisfies the equation in \eqref{e3}. Note that the
decomposition in \eqref{f2} is used in the construction of $\widehat\rho$ in \eqref{f3}.

\begin{proposition}\label{prop1}
Let $\phi$ be a holomorphic connection on $E_G\,=\, X\times G$ such that
the homomorphism $\rho$ in \eqref{e13} corresponding to $\phi$ satisfies the
following condition: the homomorphism of first cohomologies corresponding to $\rho$, namely 
the homomorphism
$$
\rho_*\, :\, H^1(X, \, TX) \, \longrightarrow\, H^1(X,\, {\rm ad}(E_G))\, ,
$$
is injective. Then
$$
{\mathbb H}^1(X,\, {\mathcal C}_\bullet) \, \supset\, {\mathcal S}(X,\, \phi)\cap \beta_\phi (H^1(X,\, TX))
\,=\, 0\, ,
$$
where ${\mathcal S}(X,\, \phi)$ is the subspace constructed in \eqref{e12} and $\beta_\phi$
is the homomorphism in \eqref{e10}.
\end{proposition}

\begin{proof}
Let
$$
q\, :\, \text{At}(E_G)\,=\, \text{ad}(E_G)\oplus TX \, \longrightarrow \, TX
$$
be the projection constructed using the decomposition in \eqref{f2}. Note that
$q$ coincides with the projection $d'p$ in \eqref{e2}; indeed, this follows from
the construction of the decomposition in \eqref{f2}. Let
\begin{equation}\label{q}
q_*\, :\, H^1(X,\, \text{At}(E_G))\,=\, H^1(X,\, \text{ad}(E_G))\oplus
H^1(X,\, TX) \, \longrightarrow \,H^1(X,\, TX)
\end{equation}
be the homomorphism of first cohomologies induced by $q$. From the equation
$$
q\circ \phi\,=\, \text{Id}_{TX}
$$
(see \eqref{e3}) it can be deduced that
\begin{equation}\label{e14}
q_*\circ\gamma\circ\beta_\phi\,=\, \text{Id}_{H^1(X, TX)}\, ,
\end{equation}
where $\beta_\phi,\, \gamma$ are the homomorphisms in \eqref{e10} and $q_*$ is constructed
in \eqref{q}. To prove \eqref{e14}, just note that $\gamma\circ\beta_\phi$ coincides with the homomorphism
of cohomologies induced by the homomorphism $\phi$.

Consider the two subspaces
\begin{equation}\label{e15}
\gamma ({\mathcal S}(X,\, \phi))\, ,\, ~ \gamma(\beta_\phi (H^1(X,\, TX)))\, \subset\,
H^1(X,\, \text{At}(E_G))
\end{equation}
of $H^1(X,\, \text{At}(E_G))$,
where $\beta_\phi,\, \gamma$ are the homomorphisms in \eqref{e10}. From \eqref{e14} it follows immediately
that the homomorphism $\gamma\circ\beta_\phi$ is injective.

Consequently, to prove the proposition
it suffices to show that
\begin{equation}\label{e16}
\gamma ({\mathcal S}(X,\, \phi))\cap (\gamma(\beta_\phi (H^1(X,\, TX))))\, =\, 0\, ,
\end{equation}
where $\gamma({\mathcal S}(X,\, \phi))$ and $\gamma(\beta_\phi (H^1(X,\, TX)))$ are the
subspaces in \eqref{e15}. At this point it might be helpful to have a look at Remark \ref{rem-pr}.

The isomonodromic deformation of the trivial connection on $E_G\, \longrightarrow\, X$ is evidently 
the trivial connection on the trivial principal $G$--bundle over the moving Riemann surface. Since 
the decomposition in \eqref{f2} is given by the trivial connection on $E_G$, from Theorem 
\ref{thm1}(4) and Theorem \ref{thm1}(3) it follows that the subspace
$$
\gamma ({\mathcal S}(X,\, \phi))\, \subset\, H^1(X,\, \text{At}(E_G))
$$
coincides with the natural subspace
\begin{equation}\label{sp}
H^1(X,\, TX)\, \subset\, H^1(X,\, \text{ad}(E_G))\oplus H^1(X,\, TX)
\end{equation}
$$
=\, H^1(X,\, \text{ad}(E_G)\oplus TX)\,=\, H^1(X,\, \text{At}(E_G))
$$
corresponding to the decomposition in \eqref{f2} (given by the trivial connection);
the subspace $H^1(X,\, TX)\, \subset\, H^1(X,\, \text{ad}(E_G))\oplus H^1(X,\, TX)$ in
\eqref{sp} consists of all $(0,\, v)$ with $v\, \in\, H^1(X,\, TX)$. The above statement follows from the
fact that for the isomonodromic deformation of the trivial connection on $E_G\, \longrightarrow\, X$, the
underlying principal $G$--bundle remains trivial (recall that the
isomonodromic deformation of the trivial connection on $E_G\, \longrightarrow\, X$ is
the trivial connection on the trivial principal $G$--bundle over the moving Riemann surface).

Therefore, using the construction of the homomorphism $\widehat{\rho}$ from $\rho$ (see \eqref{f3} --- note
that $\phi\,=\, \widehat{\rho}$ by the definition of $\rho$ given in the
statement of the proposition) it follows that the given condition --- that
$$
\rho_*\, :\, H^1(X, \, TX) \, \longrightarrow\, H^1(X,\, {\rm ad}(E_G))
$$
is injective --- implies that \eqref{e16} holds. As observed before, \eqref{e16} completes the
proof of the proposition.
\end{proof}

\begin{remark}\label{rem-pr}
Since Proposition \ref{prop1} is the key tool here, we would make some clarifying comments 
on the proof of it. Let $A,\, B$ be finite dimensional vector spaces and $$H\, :\, A\, 
\longrightarrow\, B$$ a linear map. Let $S_1,\, S_2$ be subspaces of $A$ such that the 
homomorphism $$H\vert_{S_2}\, :\, S_2\, \longrightarrow\, B$$ is injective (in other words, 
$S_2\bigcap \text{kernel}(H)\,=\, 0$). Now, if $H(S_1)\bigcap H(S_2)\,=\, 0$, then it is 
straight-forward to check that $S_1\bigcap S_2\,=\, 0$. In the proof of Proposition 
\ref{prop1}, set $$A\,=\, {\mathbb H}^1(X,\, {\mathcal C}_\bullet),\,\ B\,=\, H^1(X,\, 
\text{At}(E_G)),\,\ S_1\,=\, {\mathcal S}(X,\, \phi),$$
$S_2\,=\, \beta_\phi (H^1(X,\, TX))$ 
and $H\,=\, \gamma$. The above condition that $H(S_1)\bigcap H(S_2)\,=\, 0$ coincides with 
\eqref{e16}. The above statement that $S_1\bigcap S_2\,=\, 0$ actually coincides with the 
statement of Proposition \ref{prop1}.
\end{remark}

\section{Some examples with $G\,=\, {\rm SL}(2,{\mathbb C})$}\label{section 4}

This section focuses on the case $G\,=\, {\rm SL}(2,{\mathbb C})$. In this case we prove that the 
criterion in Proposition \ref{prop1} is satisfied at any point $(X, \phi) \in 
\text{Syst}^{\text{irred}}$ for $g\,=\, 2$ (see Proposition \ref{prop2}) and at the generic point 
in $\text{Syst}^{\text{irred}}$ for $g\,=\, 3$ (see Lemma \ref{lem1}).

\begin{proposition}\label{prop2}
Let $X$ be a compact connected Riemann surface of genus two. Set 
$G\,=\, {\rm SL}(2,{\mathbb C})$. Let $\phi$ be an irreducible
holomorphic connection on the trivial holomorphic principal $G$--bundle
$E_G\,=\, X\times G\, \longrightarrow\, X$. Then the homomorphism
$\rho$ as in \eqref{e13} corresponding to $\phi$ satisfies the
following condition: the homomorphism of cohomologies corresponding to $\rho$, namely the homomorphism
$$
\rho_*\, :\, H^1(X, \, TX) \, \longrightarrow\, H^1(X,\, {\rm ad}(E_G))\, ,
$$
is injective.
\end{proposition}

\begin{proof}
Fix a holomorphic trivialization of $E_G$. Then the adjoint vector bundle $\text{ad}(E_G)$ is the 
trivial holomorphic vector bundle $X\times {\mathfrak s}{\mathfrak l}(2, {\mathbb C})$ over $X$, 
where ${\mathfrak s}{\mathfrak l}(2, {\mathbb C})$ is the Lie algebra consisting of $2\times 2$ 
complex matrices of trace zero. So $\rho$ as in \eqref{e13} corresponding to $\phi$
$$
\rho\, :\, TX\, \longrightarrow\, \text{ad}(E_G)\,=\,
X\times {\mathfrak s}{\mathfrak l}(2, {\mathbb C})
$$
sends any $v\, \in\, T_xX$ to
$$
\left(x,\,
\begin{pmatrix}
\omega_1(x)(v) & \omega_2(x)(v)\\
\omega_3(x)(v) & -\omega_1(x)(v)
\end{pmatrix}\right)\,\in\, X\times {\mathfrak s}{\mathfrak l}(2, {\mathbb C})\, ,
$$
where $\omega_j\, \in\, H^0(X,\, K_X)$, $1\, \leq\, j\, \leq\, 3$. 

Let
\begin{equation}\label{eV}
V\,\subset\, H^0(X,\, K_X)
\end{equation}
be the linear span of $\{\omega_1,\, \omega_2,\, \omega_3\}$. The given condition that
the connection $\phi$ is irreducible implies that $\dim V\, >\, 1$. Indeed, if we assume by
contradiction that $\dim V\, \leq\, 1$, then
$$
\rho\, =\, \begin{pmatrix}
\omega_1 & \omega_2\\
\omega_3 & -\omega_1
\end{pmatrix}
 =\, \omega\otimes B\, ,
$$
where $\omega\, \in\, H^0(X,\, K_X)$ and $B\,\in\, {\mathfrak s}{\mathfrak l}(2, {\mathbb C})$ is a 
fixed element. Since the standard action on ${\mathbb C}^2$ of $B$ is reducible, the connection 
$\phi$ is reducible: a contradiction. So we have $\dim V\, >\, 1$, and, since $\dim H^0(X,\, 
K_X)\,=\, 2$, we conclude that
\begin{equation}\label{eV2}
V\,=\, H^0(X,\, K_X)\, .
\end{equation}

Let
\begin{equation}\label{The}
\Theta\, :\, H^1(X,\, TX)\otimes H^0(X,\, K_X)\, \longrightarrow\, 
H^1(X,\, TX\otimes K_X)\,=\, H^1(X,\, {\mathcal O}_X)
\end{equation}
be the natural homomorphism. For any $\mathbb C$--linear map
\begin{equation}\label{eh}
h\, :\, {\mathfrak s}{\mathfrak l}(2, {\mathbb C})\, \longrightarrow\,\mathbb C\, ,
\end{equation}
let
\begin{equation}\label{eh2}
\widetilde{h}_*\, :\, H^1(X,\, TX)\, \longrightarrow\, H^1(X,\, {\mathcal O}_X)
\end{equation}
be the homomorphism induced by the composition
$$
TX \, \stackrel{\rho}{\longrightarrow}\, \text{ad}(E_G)\,=\, {\mathcal O}_X\times_{\mathbb C}
{\mathfrak s}{\mathfrak l}(2, {\mathbb C})\, \stackrel{\text{Id}\otimes h}{\longrightarrow}\,
{\mathcal O}_X\, .
$$
For any
$$
\mu\, \in\, \text{kernel}(\rho_*)\, \subset\, H^1(X,\, TX)\, ,
$$
we evidently have
\begin{equation}\label{j1}
\widetilde{h}_* (\mu)\,=\, 0\, ,
\end{equation}
because
$\widetilde{h}_*\,=\,(\text{Id}\otimes h)_*\circ\rho_*$, where
$(\text{Id}\otimes h)_*$ is the homomorphism of cohomologies induced by
$\text{Id}\otimes h$.

{}From \eqref{j1} it follows that $\Theta(\mu\otimes V)\,=\, 0$ for
all $\mu\, \in\, \text{kernel}(\rho_*)$, where $\Theta$ is the homomorphism in \eqref{The}
and $V$ is the subspace in \eqref{eV}.
Now, since $V\,=\, H^0(X,\, K_X)$ (see \eqref{eV2}), we conclude that
\begin{equation}\label{e17}
\Theta(\mu\otimes \omega)\,=\, 0
\end{equation}
for all $\mu\, \in\, \text{kernel}(\rho_*)$ and $\omega\,\in\, H^0(X,\, K_X)$.

To complete the proof of the proposition we need to show that there is no nonzero cohomology class 
$\mu\,\in\, H^1(X,\, TX)$ that satisfies \eqref{e17} for all $\omega\, \in\, H^0(X,\, K_X)$.

Using Serre duality, it suffices to prove that the tensor product homomorphism
\begin{equation}\label{thp}
\Theta'\, :\, H^0(X,\, K_X)\otimes H^0(X,\, K_X)\, \longrightarrow\,
H^0(X,\, K^2_X)
\end{equation}
is surjective; note that $\Theta'$ is given by the dual of $\Theta$.

It is known that for a genus two Riemann surface $X$, the homomorphism
$\Theta'$ in \eqref{thp} is indeed surjective. To be somewhat self-contained, we give the outline of
an argument for it. Consider the short exact sequence of sheaves on $X\times X$
\begin{equation}\label{e18}
0\, \longrightarrow\, (p^*_1K_X)\otimes (p^*_2K_X)\otimes {\mathcal O}_{X\times X}
(-\Delta) \, \longrightarrow\, (p^*_1K_X)\otimes (p^*_2K_X)
\, \longrightarrow\, i_*K^2_X\, \longrightarrow\, 0\, ,
\end{equation}
where $p_j$ is the projection of $X\times X$ to the $j$--th factor for $j\,=\, 1,\, 2$,
and $i$ is the inclusion map of the diagonal $\Delta\, \subset\, X\times X$. Using the
short exact sequence
$$
0\, \longrightarrow\, (p^*_2K_X)\otimes{\mathcal O}_{X\times X}(-\Delta)\, \longrightarrow\,
p^*_2K_X \, \longrightarrow\, i_* K_X \, \longrightarrow\, 0
$$
we have
$$
p_{1*}((p^*_2K_X)\otimes{\mathcal O}_{X\times X}(-\Delta))\,=\, (\bigwedge\nolimits^2
(p_{1*}p^*_2K_X))\otimes (p_{1*}(i_* K_X))^*\,=\, TX\, ;
$$
this is because the homomorphism $p_{1*}p^*_2K_X\, \longrightarrow\, p_{1*}(i_* K_X)$ is
surjective as $K_X$ is base-point free. Therefore, the projection formula gives that
$$
p_{1*}((p^*_1K_X)\otimes (p^*_2K_X)\otimes{\mathcal O}_{X\times X}
(-\Delta))\,=\, {\mathcal O}_X\, .
$$
Hence we have
\begin{equation}\label{dc}
H^0(X\times X,\, (p^*_1K_X)\otimes (p^*_2K_X)\otimes{\mathcal O}_{X\times X}
(-\Delta))\,=\, H^0(X,\, {\mathcal O}_X)\, .
\end{equation}
Consider the long exact sequence of cohomologies for the short exact sequence of
sheaves in \eqref{e18}
\begin{equation}\label{dc2}
0\, \longrightarrow\, H^0(X\times X,\, (p^*_1K_X)\otimes (p^*_2K_X)\otimes{\mathcal O}_{X\times X}
(-\Delta)) 
\end{equation}
$$
\, \longrightarrow\, H^0(X\times X,\, (p^*_1K_X)\otimes (p^*_2K_X))
\,=\, H^0(X,\, K_X)^{\otimes 2}\, \stackrel{\Theta'}{\longrightarrow}\, H^0(X,\, K^2_X)\, ,
$$
where $\Theta'$ is the homomorphism in \eqref{thp}.
Since $\dim H^0(X,\, K_X)^{\otimes 2}\,=\, 4 \,=\, \dim H^0(X,\, K^2_X) + 1$, from
\eqref{dc} and \eqref{dc2} it follows that $\Theta'$ is surjective.
\end{proof}

\begin{lemma}\label{lem1}
Let $X$ be a compact connected Riemann surface of genus three which is not hyperelliptic. Set 
$G\,=\, {\rm SL}(2,{\mathbb C})$. Then there is a nonempty Zariski open subset
$\mathcal U$ of the space of all holomorphic connections 
on the trivial holomorphic principal $G$--bundle
$E_G\,=\, X\times G\, \longrightarrow\, X$ such that for any
$\phi\, \in\, \mathcal U$, the homomorphism $\rho$ in \eqref{e13} corresponding
to $\phi$ satisfies the following condition: the
homomorphism of cohomologies corresponding to $\rho$, namely the homomorphism
$$
\rho_*\, :\, H^1(X, \, TX) \, \longrightarrow\, H^1(X,\, {\rm ad}(E_G))\, ,
$$
is injective.
\end{lemma}

\begin{proof}
As in the proof of Proposition \ref{prop2}, fixing a holomorphic trivialization of $E_G$, 
identifies $\text{ad}(E_G)$ with the trivial holomorphic vector bundle $X\times {\mathfrak 
s}{\mathfrak l}(2, {\mathbb C})$ over $X$. Let $\{\omega_j\}_{j=1}^3$ be a basis of $H^0(X,\, K_X)$ 
(recall that $H^0(X,\, K_X)$ has dimension three). Define the homomorphism $\rho$
$$
\rho\, :\, TX\, \longrightarrow\, \text{ad}(E_G)\,=\,
X\times {\mathfrak s}{\mathfrak l}(2, {\mathbb C})
$$
that sends any $v\, \in\, T_xX$ to
$$
\left(x,\,
\begin{pmatrix}
\omega_1(x)(v) & \omega_2(x)(v)\\
\omega_3(x)(v) & -\omega_1(x)(v)
\end{pmatrix}\right)\,\in\, X\times {\mathfrak s}{\mathfrak l}(2, {\mathbb C})\, .
$$

As in \eqref{The}, let
$$
\Theta\, :\, H^1(X,\, TX)\otimes H^0(X,\, K_X)\, \longrightarrow\, H^1(X,\, {\mathcal O}_X)
$$
be the natural homomorphism. For any $h$ as in \eqref{eh}, the homomorphism
$\widetilde{h}_*$ as in \eqref{eh2} vanishes. Therefore, from the above construction of
$\rho$ it follows immediately that every
$$
\mu\, \in\, \text{kernel}(\rho_*)\, \subset\, H^1(X,\, TX)\, ,
$$
satisfies the equation
\begin{equation}\label{e19}
\Theta(\mu\otimes \omega)\,=\, 0
\end{equation}
for all $\omega\, \in\, H^0(X,\, K_X)$; recall that $\{\omega_j\}_{j=1}^3$ be a basis of $H^0(X,\, K_X)$.

We will show that there is no nonzero cohomology class $\mu\,\in\, H^1(X,\, TX)$
that satisfies \eqref{e19} for all $\omega\, \in\, H^0(X,\, K_X)$.

Using Serre duality, it
suffices to prove that the tensor product homomorphism
$$
\Theta'\, :\, H^0(X,\, K_X)\otimes H^0(X,\, K_X)\, \longrightarrow\,
H^0(X,\, K^2_X)
$$
is surjective. Now, Max Noether's theorem says that the homomorphism $\Theta'$ is surjective
because $X$ is not hyperelliptic \cite[p.~117]{ACGH}. Hence $\text{kernel}(\rho_*)\,=\, 0$.

The condition on a homomorphism $\rho\, :\, TX\, \longrightarrow\, \text{ad}(E_G)$ that
$$
\rho_*\, :\, H^1(X, \, TX) \, \longrightarrow\, H^1(X,\, {\rm ad}(E_G))
$$
is injective, is Zariski open (in the space of all holomorphic homomorphisms). This completes
the proof of the lemma.
\end{proof}

\section{Holomorphic connections on the trivial bundle}\label{section 5}

In this section we prove the main result of the article (Theorem \ref{thm2}) and deduce several 
consequences (Corollary \ref{cor2} and Corollary \ref{cor3}). Corollary \ref{cor2} is the main 
result in \cite{CDHL}. Corollary \ref{cor3} answers positively a question of B. Deroin.

As before, $G$ is a connected reductive affine algebraic group defined over $\mathbb C$.
In this section we further assume that $\dim G\, \geq\, 3$. As before, the Lie algebra of $G$
will be denoted by $\mathfrak g$.

As before, $X$ is a compact connected Riemann surface of genus $g$, with $g\, \geq\,2$. Given an
element
$$
\delta\, \in\, H^0(X,\, K_X)\otimes{\mathfrak g} \, ,
$$
we have an ${\mathcal O}_X$--linear homomorphism
\begin{equation}\label{e21}
M(\delta)\, :\, TX\, \longrightarrow\, {\mathcal O}_X\otimes_{\mathbb C} {\mathfrak g}
\end{equation}
that sends any $v\, \in\, T_xX$ to the contraction $\langle \delta(x),\, v\rangle\,
\in\, \mathfrak g$. Let
\begin{equation}\label{e22}
M(\delta)_*\, :\, H^1(X,\, TX)\, \longrightarrow\, H^1(X,\, {\mathcal O}_X\otimes_{\mathbb C}
{\mathfrak g})\,=\, H^1(X,\, {\mathcal O}_X)\otimes_{\mathbb C} {\mathfrak g}
\end{equation}
be the homomorphism of first cohomologies induced by the homomorphism $M(\delta)$ in \eqref{e21}. 

Notice that the homomorphism $M(\delta)$ in \eqref{e21} is similar to $\rho$ in \eqref{e13}.

\begin{lemma}\label{lem2}
Let $X$ be a compact connected Riemann surface of genus $g\, \geq\,2$
such that one of the following two holds:
\begin{enumerate}
\item $X$ is non-hyperelliptic;

\item $g\,=\, 2$.
\end{enumerate}
Then there is a nonempty Zariski open subset ${\mathcal U}\, \subset\, H^0(X,\, K_X)\otimes{\mathfrak g}$
such that for every $\delta\, \in\, {\mathcal U}$, the homomorphism $M(\delta)_*$ constructed in
\eqref{e22} is injective.
\end{lemma}

\begin{proof}
First assume that $X$ is non-hyperelliptic. Under this assumption, Theorem 1.1 of 
\cite[p.~221]{Gi} says that for a generic three dimensional subspace $W\, \subset\, H^0(X,\, 
K_X)$, the natural homomorphism

$$
\Theta_W\, :\, H^0(X,\, K_X)\otimes W\, \longrightarrow\, H^0(X,\, K^2_X)
$$
is surjective; in \cite{Gi}, the proof of this theorem is attributed to R. Lazarsfeld (see the
sentence in \cite{Gi} just after Theorem 1.1). The dual homomorphism for it
$$
\Theta^*_W\, :\, H^1(X,\, TX) \, \longrightarrow\, H^1(X,\, {\mathcal O}_X)\otimes W^*\, ,
$$
obtained using Serre duality, is injective if $\Theta_W$ is surjective.

Take any $W$ as above such that $\Theta^*_W$ is injective. Set
$$
\delta\, \in\, H^0(X,\, K_X)\otimes{\mathfrak g}
$$
to be such that the image of the homomorphism ${\mathfrak g}^*\, \longrightarrow\,
H^0(X,\, K_X)$ corresponding to $\delta$ contains $W$; note that the given condition
that $\dim G\, \geq\, 3$ ensures that such a $\delta$ exists. Then from the injectivity of $\Theta^*_W$
it follows immediately that the homomorphism $M(\delta)_*$ constructed in
\eqref{e22} is injective.

Since the condition on $\delta\, \in\, H^0(X,\, K_X)\otimes{\mathfrak g}$ that $M(\delta)_*$
is injective is actually Zariski open (in $H^0(X,\, K_X)\otimes{\mathfrak g}$), the proof of the lemma is
complete under the assumption that $X$ is non-hyperelliptic.

Next assume that $g\,=\, 2$. This case is actually covered in the proof of Proposition \ref{prop2}. 
More precisely, the proof of Proposition \ref{prop2} shows that as long as $\delta$ is not of the form 
$\omega\otimes B$, where $B\, \in\, \mathfrak g$ and $\omega\, \in\, H^0(X,\, K_X)$ are
fixed elements, the homomorphism $M(\delta)_*$ is injective.
\end{proof}

\begin{corollary}\label{cor1}
Let $X$ be a compact connected Riemann surface of genus $g\, \geq\,2$ such that one
of the following two holds:
\begin{enumerate}
\item $X$ is non-hyperelliptic;

\item $g\,=\, 2$.
\end{enumerate}
Then for the generic holomorphic connection $\phi$ on the holomorphically trivial principal
$G$--bundle $E_G\, =\, X\times G\, \longrightarrow\, X$,
$$
{\mathbb H}^1(X,\, {\mathcal C}_\bullet) \, \supset\, {\mathcal S}(X,\, \phi)\cap \beta_\phi (H^1(X,\, TX))
\,=\, 0\, ,
$$
where ${\mathcal S}(X,\, \phi)$ is the subspace constructed in \eqref{e12} and $\beta_\phi$
is the homomorphism in \eqref{e10}.
\end{corollary}

\begin{proof}
This follows from the combination of Proposition \ref{prop1} and Lemma \ref{lem2}.
\end{proof}

For any holomorphic connection $\phi$ on a holomorphic principal $G$--bundle $E_G$ (it need not
be trivial) on $X$, consider the 
monodromy representation for $\phi$. Let ${\mathbb L}(\phi)$ be the $\mathbb C$--local system on $X$ 
for the flat connection on $\text{ad}(E_G)$ induced by $\phi$.
The infinitesimal deformations of the monodromy representation are parametrized
by the elements of $H^1(X,\, {\mathbb L}(\phi))$ \cite{Go}.
The differential of the monodromy map is a homomorphism
\begin{equation}\label{e23}
{\mathcal H}(\phi)\, :\, {\mathbb H}^1(X,\, {\mathcal C}_\bullet)
\, \longrightarrow\, H^1(X,\, {\mathbb L}(\phi))
\end{equation}
(see Theorem \ref{thm1}(2)).

\begin{theorem}\label{thm2}
Let $X$ be a compact connected Riemann surface of genus $g\, \geq\,2$ such that
one of the following two holds:
\begin{enumerate}
\item $X$ is non-hyperelliptic;

\item $g\,=\, 2$.
\end{enumerate}
Then for the generic holomorphic connection $\phi$ on the holomorphically trivial principal
$G$--bundle $E_G\, =\, X\times G\, \longrightarrow\, X$, the restriction of the homomorphism
${\mathcal H}(\phi)$ in \eqref{e23} to the subspace ${\mathcal S}(X,\, \phi)\, \subset\,
{\mathbb H}^1(X,\, {\mathcal C}_\bullet)$ in \eqref{e12} is injective.
\end{theorem}

\begin{proof}
{}From Theorem \ref{thm1}(4) we know that the kernel of the homomorphism ${\mathcal H}(\phi)$
is the image of the homomorphism $\beta_\phi$ in \eqref{e10}. In view of this, the theorem is an immediate
consequence of Corollary \ref{cor1}.
\end{proof}

\begin{remark}\label{remla}
The assumption in Theorem \ref{thm principal} that $\dim_{\mathbb C} G \, \geq\, 3$ is essential. Otherwise (i.e.,
if $\dim_{\mathbb C} G \, \leq\, 2$), 
the map ${\rm Mon}$ in Theorem \ref{thm principal} fails to be an immersion for dimensional reasons.
To illustrate this, set $G$ to be the two dimensional affine algebraic torus
${\mathbb C}^*\times {\mathbb C}^*$. Then
$$
\dim_{\mathbb C} \Xi^{\rm irred}\,=\, 4g\, ,
$$
and from Lemma \ref{lemsmm} we know that
$$
\dim_{\mathbb C} {\rm Syst}^{\rm irred}\,=\, 5g-3\, .
$$
Therefore, in this case, the map
${\rm Mon}$ in Theorem \ref{thm principal} is nowhere an immersion.
\end{remark}

\begin{remark}\label{rem-re2}
The map ${\mathcal H}(\phi)$ in \eqref{e23} is a homomorphism of tangent spaces.
So the generic injectivity statement in Theorem \ref{thm2} is unrelated to whether
the points of the moduli spaces, to which the tangent spaces are associated, are
genuinely orbifold points or not. On the other hand, the injectivity statement
in Theorem \ref{thm2} is only for the generic point. Therefore, Theorem \ref{thm2}
does not say whether the injectivity statement holds at an orbifold point or not.
\end{remark}

The following result was first proved in \cite{CDHL}:

\begin{corollary}\label{cor2}
Let $X$ be a compact connected Riemann surface of genus two. Set 
$G\,=\, {\rm SL}(2,{\mathbb C})$. Let $\phi$ be an irreducible
holomorphic connection on the trivial holomorphic principal $G$--bundle
$E_G\,=\, X\times G\, \longrightarrow\, X$. Then the
restriction of the homomorphism
${\mathcal H}(\phi)$ in \eqref{e23} to the subspace ${\mathcal S}(X,\, \phi)$ in \eqref{e12}
is an isomorphism between ${\mathcal S} (X, \, \phi)$ and $H^1(X, \, {\mathbb L}(\phi))$.
\end{corollary}

\begin{proof}
Since $\text{kernel}({\mathcal H}(\phi))\,=\, \beta_\phi (H^1(X,\, TX))$ (see Theorem \ref{thm1}(4)), it follows from the
combination of Proposition \ref{prop2} and Proposition \ref{prop1} that the restriction of
${\mathcal H}(\phi)$ to ${\mathcal S}(X,\, \phi)$ is injective. Since the complex dimensions agree:
$$
\dim {\mathcal S}(X,\, \phi)\,=\, 6\,=\, \dim H^1(X,\, {\mathbb L}(\phi))\, ,
$$
injectivity implies isomorphism.
\end{proof}

\begin{remark}\label{rem-re3}
Regarding Remark \ref{rem-re2}, we note that Corollary \ref{cor2} holds for every
point $(X,\, \phi)\, \in\, {\rm Syst}^{\rm irred}$ (under the assumptions that
$g\,=\, 2$ and $G\,=\, {\rm SL}(2,{\mathbb C})$). In particular, Corollary \ref{cor2}
holds for the orbifold points of ${\rm Syst}^{\rm irred}$.
\end{remark} 

For Riemann surfaces of higher genus we have the following:

\begin{corollary} \label{cor3}
Let $X$ be a compact connected Riemann surface of genus $g\, \geq\,3$ which is not hyperelliptic. Set 
$G\,=\, {\rm SL}(2,{\mathbb C})$. 

Then for the generic holomorphic connection $\phi$ on the holomorphically trivial principal
$G$--bundle $E_G\, =\, X\times G\, \longrightarrow\, X$, the restriction of the homomorphism
${\mathcal H}(\phi)$ in \eqref{e23} to the subspace ${\mathcal S}(X,\, \phi)\, \subset\,
{\mathbb H}^1(X,\, {\mathcal C}_\bullet)$ is an isomorphism
between ${\mathcal S} (X, \, \phi)$ and $H^1(X, \, {\mathbb L}(\phi))$.
\end{corollary}

\begin{proof}
This follows directly from the injectivity statement obtained in Theorem \ref{thm2}. Indeed, here
the complex dimensions agree:
$$
\dim {\mathcal S}(X,\, \phi)\,=\, 6g-6 \,=\, \dim H^1(X,\, {\mathbb L}(\phi))\, .
$$
On the other hand, Theorem \ref{thm2} says that for the generic holomorphic connection $\phi$ on the
holomorphically trivial principal $G$--bundle $E_G\, =\, X\times G\, \longrightarrow\, X$,
the restriction of the homomorphism
${\mathcal H}(\phi)$ in \eqref{e23} to the subspace ${\mathcal S}(X,\, \phi)\, \subset\,
{\mathbb H}^1(X,\, {\mathcal C}_\bullet)$ is injective. Hence this restriction of the homomorphism
${\mathcal H}(\phi)$ is an isomorphism
between ${\mathcal S} (X, \, \phi)$ and $H^1(X, \, {\mathbb L}(\phi))$.
\end{proof}

In \cite{CDHL} the following question was asked: Is there a compact Riemann surface $X$, 
and a holomorphic connection $D$ on ${\mathcal O}^{\oplus 2}_X$, such that $D$ is 
irreducible and the image of the monodromy homomorphism for $D$ is contained in ${\rm 
SL}(2, {\mathbb R})$ \cite[p.~161]{CDHL}? Such pairs $(X,\, D)$ were constructed in 
\cite{BDH}.

\section*{Acknowledgement}

We thank the referee for helpful comments.
We are very grateful to Bertrand Deroin for formulating for us the question addressed here in 
Corollary \ref{cor3} and for interesting discussions on the subject. We are also very grateful to 
Mihnea Popa for pointing out \cite{Gi}. We thank P. Pandit for pointing out \cite{Ch1}.

This work has been supported by the French government through the UCAJEDI Investments in the 
Future project managed by the National Research Agency (ANR) with the reference number 
ANR2152IDEX201. The first-named author is partially supported by a J. C. Bose Fellowship, and 
school of mathematics, TIFR, is supported by 12-R$\&$D-TFR-5.01-0500. 


\end{document}